\documentclass[12pt,reqno]{amsart}

\usepackage{amsmath}
\usepackage{amscd,amssymb}
\usepackage{epsfig}

\usepackage{graphicx}
\usepackage{color}
\usepackage{hyperref}
\usepackage{booktabs}
\usepackage{amsfonts}
\usepackage{mathrsfs}
\usepackage{bm}
\usepackage{fancyhdr}
\usepackage{amsthm}
\usepackage{txfonts}
\usepackage{float}
\usepackage{multirow}

\def\qed{\ifmmode\square\else\nolinebreak\hfill
$\Box$\fi\par\vskip12pt}

\newtheorem{thm}{Theorem}[section]
\newtheorem{lemma}[thm]{Lemma}
\newtheorem{corollary}[thm]{Corollary}
\newtheorem{proposition}[thm]{Proposition}
\numberwithin{equation}{section}
\numberwithin{thm}{section}

\theoremstyle{definition}
\newtheorem{definition}[thm]{Definition}
\newtheorem{remark}[thm]{Remark}
\newtheorem{example}[thm]{Example}

\newcommand{\bF}{\mathbb F}

\newcommand{\bZ}{\mathbb Z}

\definecolor{Purple}{rgb}{0.5,0,0.5}

\begin{document}
\pagestyle{plain}
\parindent=19pt
\begin{titlepage}

\title{Nonexistence Results and Constructions for Signed Difference
Sets}

\begin{center}
\author{Zhiwen He$^{*}$, Jiayan Wu}\end{center}
\address{South China Normal University, Guangzhou 510631, China}
\email{ zhiwen$\_$he@zju.edu.cn }

\address{South China University of Technology, Guangzhou 510641, China}
\email{ jiayanwu@zju.edu.cn }

    \begin{abstract}
Signed difference sets (SDSs) extend ordinary difference sets by allowing negative coefficients.  We establish new nonexistence criteria and existence constructions for SDSs in finite abelian groups.  Using the classical self-conjugate-prime method and a support-refined quotient method, we derive four obstructions, including
the $C_2$-quotient, near-full-support, and small-defect obstructions. Applied cumulatively to the $67{,}823$ open group-specific cases in the database with $9\leq v\leq499$, these criteria rule out $45{,}361$ cases.  For existence, we construct an infinite family from PCP-type regular partial difference sets arising from
Desarguesian spreads and an explicit $(125,28,3)$-SDS in $C_5^3$ using quartic multiplicative characters.  These constructions settle three further open cases.
\end{abstract}

\keywords{Signed difference set; self-conjugate prime; intersection
number; partial difference set; Desarguesian spread; multiplicative
character.\\
{\bf  Mathematics Subject Classification (2020) 05E10 05E16 05E30 11T22 94A05}\\
$^*$Corresponding author}

\maketitle

\section{Introduction}

Difference sets and their generalizations provide algebraic tools for
constructing sequences, codes, and related combinatorial structures
with prescribed correlation properties.  They have found applications
in communications, information security, cryptography, and coding
theory~\cite{LiuS2018,Cusick2015,Ding}.  In particular, difference sets can be used to construct sequence
families with spectral nulls and good periodic correlation
properties~\cite{Tsai2011,YeZ2022}.

Let $G$ be a finite group of order $v$, and identify a subset of $G$ with its corresponding element in $\mathbb Z[G]$.  A signed subset of $G$ is an element
\[
D=\sum_{g\in G}a_g g,
\qquad a_g\in\{0,\pm1\}.
\]
Writing
\[
P=\{g\in G:a_g=1\},
\qquad
N=\{g\in G:a_g=-1\},
\]
we have $D=P-N$, and its support size is $k=|P|+|N|$. Following Gordon~\cite{MR4578178}, we call $D$ a $(v,k,\lambda)$-signed difference set (SDS), if
\begin{equation}
\label{Eqn_SDS}
DD^{(-1)}=(k-\lambda)0_G+\lambda G.
\end{equation}
If $N=\varnothing$, then $D$ is an ordinary $(v,k,\lambda)$-difference set(DS).
If $\lambda=0$, the SDS corresponds to a $G$-invariant weighing matrix; when $G$ is cyclic, it corresponds to a circulant weighing matrix.  Moreover,
in the cyclic case the coefficients of an SDS form a ternary periodic
sequence whose periodic autocorrelation is $k$ at the zero shift and $\lambda$ at every nonzero shift. Thus SDSs provide a common framework that includes ordinary DSs, group-invariant weighing matrices, and ternary sequences with
two-level periodic autocorrelation.

For our construction, we also use regular partial difference sets. A subset $P\subset G\setminus\{0_G\}$ with $|P|=k$ and $P^{(-1)}=P$ is called a regular {\it $(v,k,\lambda,\mu)$-partial difference set} (PDS) if 
\begin{equation}\label{Eqn_PDS}
PP^{(-1)}=\lambda P+\mu(G-P)+(k-\mu)\cdot 0_G.
\end{equation}

Gordon~\cite{MR4578178} introduced SDSs and gave constructions based on
ordinary DSs and power residues, together with a quadratic-character
construction.  Gordon's work also employed multiplier and intersection number methods to reduce the computational search space. Subsequently, He et al.~\cite{he2024new} developed further existence constructions using
Paley-type PDSs, product methods, and fourth-order cyclotomic classes.  These works produced several infinite families. However, many group-specific cases in the signed difference set database~\cite{GordonRepository} remain open, and further existence and nonexistence results are still needed.

Nonexistence questions for ordinary DSs have been studied extensively, often through character sums, cyclotomic fields, multipliers, and intersection numbers; see, for example, \cite{Hall,MR179098,turyn1965nonexistence}. Although ordinary DSs and SDSs satisfy the same group-ring identity \eqref{Eqn_SDS}, their coefficients lie in $\{0,1\}$ and $\{0,\pm1\}$, respectively. Therefore, nonexistence results for DSs do not automatically extend to SDSs: every DS is an SDS, so SDS nonexistence implies DS nonexistence, but not conversely. This paper studies both nonexistence and existence problems for SDSs. We first adapt the classical methods of self-conjugate primes and cyclotomic character sums to SDSs. The resulting obstruction uses only the nonprincipal character equation
\[
\chi(D)\overline{\chi(D)}=k-\lambda
\]
and does not distinguish between the positive and negative supports.
Nevertheless, it rules out $44{,}872$ open group-specific cases in the current database.

We develop a support-refined quotient method by recording, on each coset, both the difference between the numbers of positive and negative elements and the number of elements having nonzero coefficients. The quotient equations yield the $C_2$-quotient obstruction for even-order groups.  Combining these equations with the parity relation between the two quantities gives the near-full-support obstruction when the support misses one element, and the small-defect obstruction when it misses a small number of elements in an elementary abelian group. We state these results for abelian groups, although some arguments extend to groups with suitable normal subgroups.

For the existence problem, we use combinatorial structures whose blocks
have simple character sums.  In contrast to the Paley-PDS construction
of He et al.~\cite{he2024new}, our first construction uses PCP-type
regular PDSs from Desarguesian spreads to produce an infinite family in
elementary abelian \(2\)-groups, settling two open cases.  Our second
construction uses quartic characters on \(\mathbb F_{25}\) and
\(\mathbb F_5\) to obtain a \((125,28,3)\)-SDS in \(C_5^3\).  Unlike
Gordon's quadratic-character construction~\cite{MR4578178}, its verification involves
quartic Gauss and Jacobi sums.

The main contributions are summarized in
Table~\ref{tab:main-contributions}.  
All database counts refer to the fixed \texttt{sds.json} snapshot provided with our verification code ~\cite{SDSVerificationCode}. This snapshot contains $67{,}823$ open group-specific cases with $9\leq v\leq499$; cases with the same parameter triple but nonisomorphic ambient groups are counted separately. The nonexistence criteria are applied in the order listed in the table, and each count records only the additional exclusions among the remaining cases.
The program also generates a case-by-case list of exclusions and the numerical witnesses supporting them.

\begin{table}[htbp]
\centering
\small
\setlength{\tabcolsep}{4pt}
\renewcommand{\arraystretch}{1.15}
\caption{Summary of the main contributions of this paper.}
\label{tab:main-contributions}
\begin{tabular}{p{0.14\textwidth}p{0.10\textwidth}p{0.35\textwidth}p{0.30\textwidth}}
\hline
Type & Result & Main idea & Consequence\\
\hline
Nonexistence & Thm.~\ref{thm:self-conjugate}
& Self-conjugate prime obstruction
& Rules out \(44{,}872\) open group-specific cases.\\
Nonexistence & Thm.~\ref{cor:C2-quotient}
& \(C_2\)-quotient obstruction
& Rules out \(456\) further cases.\\
Nonexistence & Thm.~\ref{thm:near-full-support}
& Near-full-support parity obstruction
& Rules out \(29\) further cases.\\
Nonexistence & Thm.~\ref{thm:small-defect}
& Small-defect parity obstruction
& Rules out \(4\) further cases.\\
Existence & Thm.~\ref{thm:PCP-derived-SDS}
& PCP-type PDSs from Desarguesian spreads
& Gives an infinite family and settles two open cases.\\
Existence & Thm.~\ref{thm:125-28-3}
& Quartic multiplicative characters
& Settles the open \((125,28,3)\) case in \(C_5^3\).\\
\hline
\end{tabular}
\end{table}

The remainder of the paper is organized as follows.
Section~\ref{sec:preliminaries} introduces the group-ring and character
notation and recalls the required facts about cyclotomic fields.
Section~\ref{sec:self-conjugate} proves the self-conjugate prime
obstruction.  Section~\ref{sec:quotient} develops the support-refined
quotient method and derives the \(C_2\)-quotient, near-full-support,
and small-defect obstructions.  Section~\ref{Sec:construcionofSDS} presents the two existence
constructions based on PCP-type PDSs and quartic multiplicative
characters.  Section~\ref{sec_con} summarizes the results and discusses
possible directions for further research.

\end{titlepage}

\section{Preliminaries}\label{sec:preliminaries}
This section collects the notation and background used throughout the paper. We begin with the group-ring formulation of signed difference sets and record a character criterion for their existence.  We then recall the cyclotomic-field facts needed for the self-conjugate prime obstruction proved in Section~\ref{sec:self-conjugate}.

\subsection{Group rings and characters}

Let $G$ be a finite group. The integral group ring $\bZ[G]$ consists of all formal sums 
\[
\sum_{g\in G} a_g g,
\qquad a_g\in \mathbb Z.
\]
Addition and multiplication are defined by 
\begin{equation}
\sum_{g\in G}a_g g+\sum_{g\in G}b_g g=\sum_{g\in G}(a_g+b_g)g
\end{equation} 
and 
\begin{equation}
(\sum_{g\in G}a_g g)\cdot(\sum_{h \in G}b_h h)=\sum_{g,h\in G}a_gb_h(gh).
\end{equation} 
We use the standard convention of identifying a subset $U\subseteq G$ with the corresponding group-ring element 
\[
U=\sum_{u\in U}u\in\mathbb Z[G].
\]
Thus the same symbol $U$ may denote both a subset of $G$ and its
associated element in $\mathbb Z[G]$.

When $G$ is abelian, we also use its character group.  A character of $G$ is a homomorphism from $G$ to the multiplicative group of complex numbers of absolute value $1$.  The principal character is denoted by $\chi_0$, and is defined by
\[
\chi_0(x)=1
\qquad
\text{for all }x\in G.
\]
For a subset $S\subseteq G$, we write
\[
\chi(S)=\sum_{s\in S}\chi(s).
\]

Let $G$ be a finite abelian group and $D=\sum_{g\in G}a_g g$ be a $(v,k,\lambda)$-SDS. We put $n=k-\lambda$ and $s=|P|-|N|$. Applying the principal character to the defining identity Eq. \eqref{Eqn_SDS} gives
\begin{equation}
\label{eq:principal-character}
s^2=n+\lambda v
   =k+\lambda(v-1).
\end{equation}

We recall the following character criterion from \cite[Lemma~2.2]{he2024new}.   
\begin{lemma}
\label{lem:character-equation}
Let $G$ be a finite abelian group of order $v>1$, and let $D$ be a signed set of $G$ with  support size $k$. Let $\lambda$ be an integer and put $n=k-\lambda$. Then $D$ is a $(v,k,\lambda)$-SDS if and only if 
\begin{equation}\label{eqn:chichin}
    \chi(D)\overline{\chi(D)}=n
\end{equation}
for every nonprincipal character $\chi$ of $G$.
\end{lemma}

\subsection{Cyclotomic preliminaries}
\label{subsec:cyclotomic-preliminaries}

In this subsection, we recall the basic facts on cyclotomic fields, prime decomposition, and Frobenius automorphisms that will be needed in the proof of our nonexistence criterion.  All the number-theoretic results stated below are standard. For the reader's convenience, we recall the relevant definitions and precise statements; for proofs and further background, see \cite[Chapters~12 and~13]{ireland1990classical}.

For a positive integer $m$, let $\zeta_m=e^{2\pi i/m}$ be a primitive $m$-th root of unity, and put $K_m=\mathbb Q(\zeta_m)$. We denote by $\mathcal O_{K_m}$ the ring of integers of $K_m$, that is, the set of all elements of $K_m$ which are roots of
monic polynomials with coefficients in $\mathbb Z$. We shall use the Galois structure of $K_m/\mathbb Q$. Recall that $\operatorname{Gal}(K_m/\mathbb Q)$ is the group of field automorphisms of $K_m$ that fix $\mathbb Q$ pointwise. The standard cyclotomic-field isomorphism is
\begin{equation}
\label{eq:gal-cyclotomic}
\operatorname{Gal}(K_m/\mathbb Q)
\cong
(\mathbb Z/m\mathbb Z)^\times.
\end{equation}
More precisely, for every $a\in(\mathbb Z/m\mathbb Z)^\times$
, the corresponding
automorphism $\sigma_a$ is determined by $\sigma_a(\zeta_m)=\zeta_m^a$. 
Conversely, every element of $\operatorname{Gal}(K_m/\mathbb Q)$ is of this form. In particular, complex conjugation on $K_m$ is the automorphism $\sigma_{-1}$, where $\sigma_{-1}(\zeta_m)=\zeta_m^{-1}$. 

Let $p$ be a rational prime.  A prime ideal $\mathfrak P$ of $\mathcal O_{K_m}$ is said to
\emph{lie above $p$} if $\mathfrak P\cap\mathbb Z=p\mathbb Z$. Equivalently, $\mathfrak P$ occurs in the prime ideal
factorization of the extended ideal $p\mathcal O_{K_m}$. More generally, one may write
\[
p\mathcal O_{K_m}
=
\mathfrak P_1^{e_1}\cdots\mathfrak P_g^{e_g},
\]
where the $\mathfrak P_i$'s are the distinct prime ideals
of $\mathcal O_{K_m}$ lying above $p$. The integer $e_i$ is called the ramification index of $\mathfrak P_i$ over $p$. We say that $p$ is \emph{unramified} in $K_m/\mathbb Q$
if $e_1=\cdots=e_g=1$. For cyclotomic fields, one has the standard ramification criterion 
\begin{equation}
\label{eq:unramified-cyclotomic}
p\nmid m
\quad\Longrightarrow\quad
p\text{ is unramified in }K_m/\mathbb Q.
\end{equation}
For a prime ideal $\mathfrak P$ of $\mathcal O_{K_m}$ lying above $p$,
its decomposition group is defined by
\[
\mathcal D(\mathfrak P\mid p)
=
\{\sigma\in \operatorname{Gal}(K_m/\mathbb Q):
\sigma(\mathfrak P)=\mathfrak P\}.
\]
When $p\nmid m$, the decomposition group is generated by the
Frobenius automorphism $\sigma_p$, namely,
\begin{equation}
\label{eq:decomp-frobenius}
\mathcal D(\mathfrak P\mid p)
=
\langle \sigma_p\rangle,
\qquad
\sigma_p(\zeta_m)=\zeta_m^p.
\end{equation}

The following standard consequence of the Frobenius description will
be used in the proof of the self-conjugate prime obstruction.

\begin{lemma}
\label{lem:self-conjugate-prime}
Let $m>1$, let $p$ be a rational prime with $p\nmid m$,
and let $\mathfrak P$ be a prime ideal of $\mathcal O_{K_m}$ lying above $p$.
If $p^j\equiv-1\pmod m$ for some integer $j\ge1$, then $\overline{\mathfrak P}=\mathfrak P$. 
\end{lemma}

\begin{proof}
By \eqref{eq:decomp-frobenius}, $\mathcal D(\mathfrak P\mid p)=\langle\sigma_p\rangle$. 
Since $p^j\equiv-1\pmod m$, we have $\zeta_m^{p^j}=\zeta_m^{-1}$. Therefore
\[
\sigma_p^j(\zeta_m)
=
\zeta_m^{p^j}
=
\zeta_m^{-1}
=
\sigma_{-1}(\zeta_m).
\]
Since $K_m=\mathbb Q(\zeta_m)$, an automorphism of $K_m$ is determined by its value on $\zeta_m$. Hence $\sigma_p^j=\sigma_{-1}$. It follows that $\sigma_{-1}\in \mathcal D(\mathfrak P\mid p)$. By the definition of the decomposition group, $\sigma_{-1}(\mathfrak P)=\mathfrak P$. Since \(\sigma_{-1}\) is complex conjugation, this means exactly
that $\overline{\mathfrak P}=\mathfrak P$. 
\end{proof}

\section{A self-conjugate prime obstruction for signed difference sets}
\label{sec:self-conjugate}
The use of self-conjugate primes and cyclotomic character sums is classical in the theory of ordinary DS; see, for example,
Turyn~\cite{MR179098}.  In this section we prove the corresponding character-norm obstruction for SDS. This is the signed analogue of a classical character-norm obstruction for ordinary DS, and its proof depends only on the character equation \eqref{eqn:chichin}. 

For a rational prime $p$ and a positive integer $n$, we denote
by $v_p(n)$ the $p$-adic valuation of $n$, namely the largest
nonnegative integer $e$ such that $p^e\mid n$. Thus, if $n=p^e n_0$ and $p\nmid n_0$, then $v_p(n)=e$. We denote by $\mathcal O_{K_m}$ the ring of integers of $K_m$, that is, the set of algebraic integers contained in $K_m$. Throughout, $C_v$ denotes the cyclic group of order $v$.

\begin{lemma}
\label{lem:character-order}
Let $G$ be a finite abelian group with exponent $E=\exp(G)$. If $m\mid E$, then $G$ has a character of order $m$.
\begin{proof}
By the structure theorem for finite abelian groups, one may write
\[
G
\cong
C_{d_1}\times C_{d_2}\times\cdots\times C_{d_r},
\qquad
d_1\mid d_2\mid\cdots\mid d_r.
\]
The exponent of $G$ is $E=d_r$. Since $m\mid E$, the cyclic group $C_E$ has a character
of order $m$.  Composing such a character with the natural
projection 
\[
\pi:G\longrightarrow C_E
\]
gives a character of $G$ of order $m$.
\end{proof}
\end{lemma}

\begin{thm}
\label{thm:self-conjugate}
Let $G$ be a finite abelian group with exponent $E=\exp(G)$, and let $D$ be a $(v,k,\lambda)$-SDS in $G$. Put $n=k-\lambda>0$. Suppose that there exist an integer $m>1$, a rational prime $p$, and an integer $j\ge1$ such that $m\mid E$ and $p^j\equiv-1\pmod m$. Then $v_p(n)$ is even.

Consequently, if $m\mid E$, $p^j\equiv-1\pmod m$ and $v_p(k-\lambda)\ \text{is odd}$, 
then no $(v,k,\lambda)$-SDS exists in $G$.
\end{thm}

\begin{proof}
Assume that $D$ is a $(v,k,\lambda)$-SDS in $G$. Since $m\mid E=\exp(G)$, 
Lemma~\ref{lem:character-order} implies that $G$ has a
character $\chi$ of order $m$.
Fix a primitive $m$-th root of unity $\zeta_m$.
Since $\chi$ has order $m$, all values of $\chi$ are
$m$-th roots of unity. Hence
\[
\alpha:=\chi(D)\in\mathbb Z[\zeta_m].
\]
Since \(\zeta_m\) is an algebraic integer, every element of
\(\mathbb Z[\zeta_m]\) is an algebraic integer. Therefore $\alpha\in\mathcal O_{K_m}$. 
By Lemma~\ref{lem:character-equation},
\begin{equation}
\label{eq:self-norm}
\alpha\overline{\alpha}=n.
\end{equation}

Let $e=v_p(n)$ and write $n=p^e n_0$ and $p\nmid n_0$. We first note that $p\nmid m$. Indeed, if $p\mid m$, then the congruence $p^j\equiv-1\pmod m$ would imply $m\mid p^j+1$, 
and hence in particular $p\mid p^j+1$, which is impossible. Therefore $p$ is unramified in $K_m=\mathbb Q(\zeta_m)$ by
\eqref{eq:unramified-cyclotomic}.  Let $\mathfrak P$ be any prime ideal of $\mathcal O_{K_m}$ lying above $p$. Since $p$ is unramified, the factorization $p\mathcal O_{K_m}
=
\mathfrak P_1\cdots\mathfrak P_g$ 
contains no repeated prime ideal. Consequently, $p^e\mathcal O_{K_m}
=
\mathfrak P_1^e\cdots\mathfrak P_g^e$.

Moreover, no prime ideal lying above $p$ divides $n_0\mathcal O_{K_m}$.   Indeed, if $\mathfrak P\mid n_0\mathcal O_{K_m}$, then $n_0\in\mathfrak P$. Since $n_0\in\mathbb Z$ and $\mathfrak P\cap\mathbb Z=p\mathbb Z$, 
this would imply $p\mid n_0$, contrary to the choice of $n_0$.
It follows that the exponent of $\mathfrak P$ in the
factorization of the principal ideal $(n)=n\mathcal O_{K_m}$ is exactly $e$. In other words, 
\begin{equation}
\label{eq:P-valuation-n}
v_{\mathfrak P}((n))
=
e
=
v_p(n).
\end{equation}

On the other hand, the congruence $p^j\equiv-1\pmod m$ and Lemma~\ref{lem:self-conjugate-prime} imply that
\begin{equation}
\label{eq:P-fixed-conjugation}
\overline{\mathfrak P}=\mathfrak P.
\end{equation}

Taking principal ideals in \eqref{eq:self-norm}, we obtain $(\alpha)(\overline{\alpha})=(n)$. 
Taking the $\mathfrak P$-adic valuation of both sides gives
\begin{equation}
\label{eq:P-valuation-alpha}
v_{\mathfrak P}((\alpha))
+
v_{\mathfrak P}((\overline{\alpha}))
=
v_{\mathfrak P}((n)).
\end{equation}

We now explain why the two terms on the left-hand side are equal.
Complex conjugation sends the prime ideal factorization of an ideal
to the prime ideal factorization of its conjugate. Thus, for every
ideal $\mathfrak a\subseteq\mathcal O_{K_m}$, $v_{\mathfrak P}(\overline{\mathfrak a})
=
v_{\overline{\mathfrak P}}(\mathfrak a)$. Applying this to $\mathfrak a=(\alpha)$ and using \eqref{eq:P-fixed-conjugation}, we obtain $v_{\mathfrak P}((\overline{\alpha}))
=
v_{\overline{\mathfrak P}}((\alpha))
=
v_{\mathfrak P}((\alpha))$. Therefore \eqref{eq:P-valuation-alpha} becomes $v_{\mathfrak P}((n))
=
2v_{\mathfrak P}((\alpha))$. Hence $v_{\mathfrak P}((n))$ is even.

Finally, by \eqref{eq:P-valuation-n}, $v_p(n)=v_{\mathfrak P}((n))$, 
and therefore $v_p(n)$ is even.  This proves the theorem.
\end{proof}

\begin{corollary}
\label{cor:q-plus-one}
Let $q$ be a rational prime.
If $q+1\mid v$ and $v_q(k-\lambda)$ is odd, then there is no cyclic $(v,k,\lambda)$-signed difference set. In particular, if $3\mid v$ and $v_2(k-\lambda)$ is odd, then no cyclic $(v,k,\lambda)$-signed difference set exists.
\begin{proof}
Let $G=C_v$. Since $\exp(G)=v$, the assumption $q+1\mid v$ implies $q+1\mid \exp(G)$. Taking $m=q+1$, 
we have $q\equiv -1\pmod m$. Hence Theorem~\ref{thm:self-conjugate}, with $p=q$ and $j=1$, shows that $v_q(k-\lambda)$ must be even. Therefore, if $v_q(k-\lambda)$ is odd, no such cyclic signed difference set can exist. The final assertion follows by taking $q=2$.
\end{proof}
\end{corollary}

\begin{example}
In the current signed difference set database, there are $67{,}823$ open group-specific entries, with group order $v$ ranging from $9$ to $499$.  Applied to these open entries,
Theorem~\ref{thm:self-conjugate} rules out $44{,}872$ of them.
Here are three small illustrative cases.

\begin{itemize}
\item Consider $\operatorname{SDS}(37,13,1,[37])$. Here $G=C_{37}$, so $E=\exp(G)=37$. Moreover,
\[
n=k-\lambda=12,
\qquad
3^9\equiv -1\pmod{37},
\qquad
v_3(12)=1.
\]
Taking $m=37$ and $p=3$, Theorem~\ref{thm:self-conjugate}
excludes this cyclic SDS.

\item Consider $\operatorname{SDS}(51,24,6,[51])$. Here
\[
n=k-\lambda=18,
\qquad
2+1=3\mid 51,
\qquad
v_2(18)=1.
\]
Thus Corollary~\ref{cor:q-plus-one}, with $q=2$, excludes this cyclic
SDS.

\item Consider $\operatorname{SDS}(48,37,4,[48])$. Here
\[
n=k-\lambda=33,
\qquad
3+1=4\mid 48,
\qquad
v_3(33)=1.
\]
Thus Corollary~\ref{cor:q-plus-one}, with $q=3$, excludes this
cyclic SDS.
\end{itemize}
\end{example}

\section{Support-refined intersection-number obstructions}
\label{sec:quotient}

Intersection-number methods are a standard tool in difference-set
theory, and they have also been used for SDS; see, for example, Gordon~\cite[Lemma~5.2]{MR4578178}. Here we use the natural quotient map $\mathbb Z[G]\to \bZ[G/H]$. For each coset of $H$, we record both the signed intersection number and the number of nonzero coefficients of $D$. These two quantities have the same parity. The quotient identities yield the $C_2$-quotient obstruction, while combining them with the parity information gives the near-full-support and small-defect obstructions.

\subsection{Support-refined intersection numbers}\label{subsetction4.1}
In this subsection we introduce the quotient intersection numbers used
throughout the section.  For each coset of a subgroup $H\leq G$, we
record both the signed intersection number and the support size of
$D$. These two quantities coincide for ordinary difference
sets, but differ in the signed setting; their parity relation is the
extra information used in the obstructions below. 

Throughout this section, let $D=P-N$ be a $(v,k,\lambda)$-SDS in a finite abelian group $G$, and put $n=k-\lambda$ and $s=|P|-|N|$. Let $H\leq G$, and let
\[
\pi:G\longrightarrow Q:=G/H
\]
be the canonical quotient map.  For $x\in Q$, denote by $C_x=\pi^{-1}(x)$ the corresponding coset of $H$.  We define
\[
A_x
=
|P\cap C_x|-|N\cap C_x|
\]
and
\[
K_x
=
|P\cap C_x|+|N\cap C_x|.
\]
Thus $A_x$ is the signed intersection number of $D=P-N$ with
$C_x$, while $K_x$ is the support size of $D$ on $C_x$. Clearly, $\sum_{x\in Q}A_x=|P|-|N|=s$ and $\sum_{x\in Q}K_x=|P|+|N|=k$. 

\begin{lemma}
\label{lem:support-refined}
Let \(h=|H|\).  Then for every $x\in Q$,
\begin{equation}
\label{eq:quotient-parity}
|A_x|\leq K_x\leq h,
\qquad
K_x\equiv A_x\pmod 2.
\end{equation}
Furthermore,
\begin{equation}
\label{eq:quotient-square}
\sum_{x\in Q}A_x^2
=
n+\lambda h,
\end{equation}
and, for every nonzero element $y\in Q$,
\begin{equation}
\label{eq:SR-correlation}
\sum_{x\in Q}A_xA_{x-y}
=
\lambda h.
\end{equation}
\end{lemma}

\begin{proof}
By the definitions of $A_x$ and $K_x$, $K_x-A_x=2|N\cap C_x|$, 
so $K_x\equiv A_x\pmod2$.  Also,
\[
|A_x|
\le |P\cap C_x|+|N\cap C_x|
=K_x
\le |C_x|=h.
\]
This proves Eq. \eqref{eq:quotient-parity}. 

Extend the quotient map linearly to a group-ring homomorphism $\pi:\mathbb Z[G]\longrightarrow\mathbb Z[Q]$. By the definition of $A_x$, we have $\pi(D)=\sum_{x\in Q}A_xx$. Since every coset of $H$ has cardinality $h$, $\pi(G)=h\sum_{x\in Q}x$. Applying $\pi$ to the signed difference set identity $DD^{(-1)}=n0_G+\lambda G$ gives
\begin{equation}
\label{eq:SR-master}
\left(
\sum_{x\in Q}A_xx
\right)
\left(
\sum_{x\in Q}A_x(-x)
\right)
=
n0_Q+\lambda h\sum_{x\in Q}x.
\end{equation}

Comparing the coefficient of $0_Q$ in
Eq. \eqref{eq:SR-master} gives $\sum_{x\in Q}A_x^2=n+\lambda h$, which proves Eq. \eqref{eq:quotient-square}.
For any nonzero $y\in Q$, comparison of the coefficient of $y$
gives $\sum_{x\in Q}A_xA_{x-y}=\lambda h$, which proves Eq. \eqref{eq:SR-correlation}. This completes the proof.
\end{proof}

\begin{thm}
\label{cor:C2-quotient}
Let $G$ be a finite abelian group of even order $v$. For any $(v,k,\lambda)$-SDS in $G$, $k-\lambda$ is a perfect square.
\end{thm}

\begin{proof}
Since $G$ is abelian and has even order, the structure theorem for
finite abelian groups gives a cyclic factor of even order.  Projecting
onto this factor and then reducing modulo $2$ gives a surjective
homomorphism $G\to C_2$.  Let $H$ be its kernel. Write $G/H=\{0,g\}$, and let $A_0,A_g$ be the corresponding signed intersection
numbers. By Eq. \eqref{eq:quotient-square}, $A_0^2+A_g^2=n+\lambda|H|$, where $Q=G/H=\{0_Q,g\}$. Applying Eq. \eqref{eq:SR-correlation} with the nonzero element $y=g$ gives  $2A_0A_g=\lambda |H|$. Then we have
\[
n=A_0^2+A_g^2-2A_0A_g=(A_0-A_g)^2.
\]
Since $n=k-\lambda$, the result follows.
\end{proof}

\begin{example}
\label{ex:C2-quotient-database}
After the self-conjugate obstruction of Theorem~\ref{thm:self-conjugate} has been applied, there remain $22{,}951$ open group-specific cases in the signed difference set
database. Among these remaining cases, Theorem~\ref{cor:C2-quotient} rules out $456$ further cases.

For instance, consider the parameter set $\operatorname{SDS}(28,10,2,[2,14])$. Here $G\cong C_2\times C_{14}$ has even order, and $k-\lambda=10-2=8$, which is not a perfect square.  Hence
Theorem~\ref{cor:C2-quotient} excludes this parameter set.
On the other hand, it is not excluded by Theorem~\ref{thm:self-conjugate}: the only prime divisor of $k-\lambda=8$ with odd valuation is $2$, and for the odd divisor $7\mid\exp(G)=14$, the powers of $2$ modulo $7$ cycle through $2,4$ and $1$, so no power is congruent to $-1\pmod7$.
\end{example}

\subsection{A near-full-support parity obstruction}
We now consider SDS whose support misses exactly
one group element, so that \(k=v-1\).  This case has extra information
coming from the location of the unique zero coefficient.  After a
translation, we may put this zero at the identity.  Then, relative to a
subgroup \(H\), the identity coset has one missing support element,
whereas the other cosets are fully supported.  This gives a fixed
parity pattern for the support sizes on the cosets.  Using
\(K_x\equiv A_x\pmod2\), this parity pattern gives a congruence
obstruction modulo \(8\), including cases not detected by the
self-conjugate prime obstruction or the \(C_2\)-quotient obstruction.

\begin{thm}
\label{thm:near-full-support}
Let $G$ be a finite abelian group of order $v=qh$, and suppose that $G$ has a subgroup $H$ of odd order $h$. If $G$ has a $(v,v-1,\lambda)$-SDS, then
\begin{equation}
\label{eq:near-full-parity}
v-1+\lambda(h-1)
\equiv
q-1
\quad\text{or}\quad
q+3
\pmod8.
\end{equation}
\end{thm}

\begin{proof}
Suppose that $D$ is a \((v,v-1,\lambda)\)-signed difference set. Since $k=v-1$, exactly one element of $G$ has coefficient zero in $D$
. After translating $D$, if necessary, we may assume that the unique zero coefficient occurs at $0_G$.  Indeed, for any $a\in G$, the
translate $aD$ satisfies 
\[
(aD)(aD)^{(-1)}=aDD^{(-1)}a^{-1}=DD^{(-1)}.
\]
 Thus translating $D$ preserves the signed difference set parameters.

Let $A_x$ and $K_x$, $x\in G/H$, be defined as above. Since the identity coset $H$ contains the unique zero coefficient, $K_0=h-1$. Every other coset contains no zero coefficient and therefore $K_x=h$ for any $x\neq0$. Since $h$ is odd, $K_0=h-1$ is even, whereas $K_x=h$ is odd for every $x\neq0$. By Eq. \eqref{eq:quotient-parity}, it follows that $A_0\equiv0\pmod2$, whereas $A_x\equiv1\pmod2$ for any $x\neq0$. The square of every odd integer is congruent to $1\pmod8$, while the square of every even integer is congruent to either $0$ or
$4\pmod8$. Since there are $q-1$ nonidentity cosets, we obtain
\begin{equation}
\label{eq:near-full-square-mod8}
\sum_{x\in G/H}A_x^2
\equiv
q-1
\quad\text{or}\quad
q+3
\pmod8.
\end{equation}

On the other hand, by Eq. \eqref{eq:quotient-square}, $\sum_{x\in G/H}A_x^2=n+\lambda h$. Since $k=v-1$, $n=k-\lambda=v-1-\lambda$. Thus $n+\lambda h=v-1+\lambda(h-1)$. Combining this identity with Eq. \eqref{eq:near-full-square-mod8} proves Eq. \eqref{eq:near-full-parity}.
\end{proof}

The preceding theorem yields an infinite nonexistence family for
abelian $p$-groups.

\begin{corollary}
\label{cor:abelian-p-group-family}
Let $p$ be a prime satisfying $p\equiv3\pmod4$, and let $m\geq2$ be even. Let $G$ be any finite abelian group of order $p^m$. If $\lambda$ is even, then there is no $(p^m,p^m-1,\lambda)$-SDS in $G$.
\end{corollary}

\begin{proof}
By the structure theorem for finite abelian groups, we may write
\[
G\cong C_{a_1}\times\cdots\times C_{a_r},
\]
where each $a_i$ is a positive power of $p$.  Since $|G|=p^m$ with $m\ge2$, at least one factor is nontrivial; relabel so that $a_1>1$.  Projecting onto the first factor and then reducing modulo $p$ gives a surjective homomorphism
\[
G\longrightarrow C_{a_1}\longrightarrow C_p.
\]
Let $H$ be its kernel. Therefore $|H|=p^{m-1}$ and $[G:H]=p$. Applying Theorem~\ref{thm:near-full-support} with $q=p$ and $h=p^{m-1}$, 
we obtain the necessary condition
\begin{equation}
\label{eq:p-group-near-full}
p^m-1+\lambda(p^{m-1}-1)
\equiv
p-1
\quad\text{or}\quad
p+3
\pmod8.
\end{equation}

Since $p$ is odd, $p^2\equiv1\pmod8$. As $m$ is even, $p^m\equiv1\pmod8$, while $m-1$ is odd and hence $p^{m-1}\equiv p\pmod8$. Therefore the left-hand side of Eq. \eqref{eq:p-group-near-full} is congruent to $\lambda(p-1)\pmod8$. 

Now $p\equiv3\pmod4$, so either $p\equiv3$ or $7\pmod8$.
 Since $\lambda$ is even, in either case
\[
\lambda(p-1)
\equiv0
\quad\text{or}\quad
4
\pmod8.
\]
On the other hand, $p-1$ and $p+3$ are congruent to $2$ and $6$ modulo $8$, in some order. This contradicts Eq. \eqref{eq:p-group-near-full}.
\end{proof}

\begin{example}
\label{ex:near-full-database}
After the self-conjugate obstruction and the \(C_2\)-quotient
obstruction above have been applied, \(22{,}495\) open cases remain.
Among these remaining cases, Theorem~\ref{thm:near-full-support}
rules out \(29\) further group-specific open cases with
\[
v\in
\{49,69,81,133,165,209,261,273,321,341,469,497\}.
\]
Thus these are genuinely new exclusions at this stage of the
argument. We give three illustrative examples.  

\begin{itemize}
\item The parameter set $(49,48,2)$ is covered directly by Corollary~\ref{cor:abelian-p-group-family}.  Indeed,
\[
49=7^2,\qquad 7\equiv3\pmod4,
\qquad \lambda=2 \text{ is even}.
\]
Hence there is no $(49,48,2)$-SDS in any abelian group of order $49$.

\item The parameter sets $(81,80,4)$ and $(81,80,44)$ are also covered directly by Corollary~\ref{cor:abelian-p-group-family}.  Here
\[
81=3^4,\qquad 3\equiv3\pmod4,
\]
and both $\lambda=4$ and $\lambda=44$ are even. Hence no $(81,80,4)$-SDS or $(81,80,44)$-SDS exists in any abelian group of order $81$.

\item Theorem~\ref{thm:near-full-support} is more general than
Corollary~\ref{cor:abelian-p-group-family}.  For example, consider
the cyclic case $(209,208,12)$ in $C_{209}$.  Since $209=19\cdot 11$, we may take $h=11$ and $q=19$.  Then
\[
v-1+\lambda(h-1)
=
208+12(11-1)
=
328
\equiv0\pmod8.
\]
But
\[
q-1=18\equiv2\pmod8,
\qquad
q+3=22\equiv6\pmod8.
\]
Thus Theorem~\ref{thm:near-full-support} rules out this parameter
set, although it is not covered by the \(p^m\)-group corollary.
\end{itemize}

These examples are new at this stage.  The $C_2$-quotient obstruction does not apply because the group orders are odd, and the
self-conjugate prime obstruction does not exclude the displayed
parameters. Thus the exclusions come from the near-full-support
parity obstruction.
\end{example}

\subsection{A small-defect parity obstruction}
\label{subsec:small-defect}

The near-full-support obstruction treats the case in which the
support of an SDS misses exactly one group element.  For elementary
abelian groups, the same parity argument can be extended when the
number of zero coefficients is small.  This yields the following
general obstruction.

\begin{thm}
\label{thm:small-defect}
Let $G=C_p^m$, where $p$ is an odd prime, and let $D$ be a $(p^m,k,\lambda)$-SDS in $G$. Put $d=p^m-k$. If $1\le d\le m$, then
\begin{equation}
\label{eq:small-defect}
k+\lambda(p^{m-1}-1)
\equiv
\begin{cases}
p \pmod 8,
   & \text{if \(d\) is even},\\[1mm]
p-1\ \text{or}\ p+3 \pmod 8,
   & \text{if \(d\) is odd}.
\end{cases}
\end{equation}
\end{thm}

\begin{proof}
Since $d=p^m-k$, the signed difference set $D$ has exactly $d$ zero coefficients. After translating $D$, if necessary, we may assume that one of
the zero coefficients occurs at $0_G$. Let $z_1,\ldots,z_{d-1}$ be the remaining zero elements.  Since $G=C_p^m$ is an $m$-dimensional vector space over $\mathbb F_p$, the subspace $\langle z_1,\ldots,z_{d-1}\rangle$ has dimension at most $d-1\le m-1$. Then the subspace $\langle z_1,\ldots,z_{d-1}\rangle$ is contained in an $(m-1)$-dimensional subspace $H$ of $G$.
Thus $H$ is a hyperplane containing $0_G,z_1,\ldots,z_{d-1}$, and hence all $d$ zero coefficients of $D$.

Since $|H|=p^{m-1}$, $G/H\cong C_p$. 
Let $A_0,\ldots,A_{p-1}$ and $K_0,\ldots,K_{p-1}$ be the signed intersection numbers and support sizes defined in Subsection~\ref{subsetction4.1},
with $H$ taken as the identity coset. Since all $d$ zero coefficients lie in $H$, we have $K_0=p^{m-1}-d$, 
whereas every other coset contains no zero coefficient, and hence $K_j=p^{m-1}$ for any $1\le j\le p-1$. 

Since $p$ is odd, $p^{m-1}$ is odd. Suppose first that $d$ is even. Then $p^{m-1}-d$ is also odd. Hence all the support sizes $K_j$ are odd.
By Eq. \eqref{eq:quotient-parity}, all the signed intersection numbers
$A_j$ are odd. Therefore $A_j^2\equiv1\pmod8$ for every $j$, and consequently 
\[
\sum_{j=0}^{p-1}A_j^2\equiv p\pmod8.
\]
Suppose next that $d$ is odd.  Then $K_0=p^{m-1}-d$ is even, while $K_j=p^{m-1}$ is odd for $1\le j\le p-1$. Thus, by Eq.~\eqref{eq:quotient-parity}, $A_0\equiv0\pmod2$ and $A_j\equiv1\pmod2$ for any $1\le j\le p-1$. It follows that $A_0^2\equiv0\ \text{or}\ 4\pmod8$, 
whereas $A_j^2\equiv1\pmod8$ for any $1\le j\le p-1$. Hence
\[
\sum_{j=0}^{p-1}A_j^2
\equiv
p-1
\quad\text{or}\quad
p+3
\pmod8.
\]
Finally, by Eq.~\eqref{eq:quotient-square}, $\sum_{j=0}^{p-1}A_j^2=n+\lambda p^{m-1}$. Since $n=k-\lambda$, the desired congruence follows.
\end{proof}

\begin{example}
\label{ex:small-defect-database}
After the self-conjugate obstruction, the $C_2$-quotient
obstruction, and the near-full-support obstruction have been applied, there remain $22{,}466$ open group-specific cases in the signed difference set database.  Theorem~\ref{thm:small-defect} rules out $4$ further cases, namely
\[
(27,25,16,[3,3,3])
\]
and
\[
(343,340,2,[7,7,7]),\quad
(343,340,72,[7,7,7]),\quad
(343,340,98,[7,7,7]).
\]
In the first case, $p=3$, $m=3$ and $d=27-25=2$. Since $d$ is even, Theorem~\ref{thm:small-defect} requires $k-\lambda+\lambda p^{m-1}
\equiv p\pmod8$. However, $25-16+16\cdot3^2=153\equiv1\pmod8$, 
whereas $p=3$.  Hence no $(27,25,16)$-SDS exists in $C_3^3$. For the three cases in $C_7^3$, $p=7$, $m=3$ and $d=343-340=3$. 
Since $d$ is odd, the allowed residues are $p-1\equiv6\pmod8$ or $p+3\equiv2\pmod8$. But for each of $\lambda=2,72,98$, $340-\lambda+\lambda7^2
=340+48\lambda
\equiv4\pmod8$. Thus all three cases are excluded by Theorem~\ref{thm:small-defect}.

These are new exclusions at this stage.  The $C_2$-quotient
obstruction does not apply, because the group orders $27$ and $343$ are odd.  The near-full-support obstruction does not apply
either, since here $27-25=2$ and $343-340=3$, so the support misses more than one group element. They are also not excluded by the self-conjugate prime obstruction.
For $(27,25,16,[3,3,3])$, one has $k-\lambda=25-16=9=3^2$, so the only prime divisor of $k-\lambda$ has even valuation. For the three cases in $C_7^3$, the values of $k-\lambda$ are $338=2\cdot13^2$, $268=4\cdot67$ and $242=2\cdot11^2$. 
The prime divisors occurring to odd exponent are congruent to $2$ or $4$ modulo $7$, and their powers modulo $7$ cycle through $2,4$ and $1$. Thus no such power is congruent to \(-1\equiv6\pmod7\), so the
self-conjugate obstruction does not exclude these cases.
\end{example}

\begin{remark}
\label{rem:small-defect-condition}
The hypothesis $d\le m$ is used only to ensure that, after translating $D$, all zero coefficients can be placed inside a hyperplane. The proof actually applies whenever the set of zero coefficients is contained in a coset of an index-$p$ subgroup.
\end{remark}

\section{Constructions of SDSs}\label{Sec:construcionofSDS}
\subsection{A PCP-derived construction from a Desarguesian spread}

We first recall the Desarguesian spread construction and the associated PDSs of partial congruence partition type. Let $p$ be a prime and $m$ be a positive integer. Regard $V=\mathbb F_{p^m}^2$ as a $2$-dimensional vector space over $\mathbb F_{p^m}$.

\begin{definition}
\label{def:Desarguesian-spread}
Define 
\[
\mathcal S
=
\bigl\{
U_a : a\in\mathbb F_{p^m}
\bigr\}
\cup
\{U_\infty\},
\]
where
\[
U_a
=
\{(x,ax):x\in\mathbb F_{p^m}\},
\qquad
U_\infty
=
\{(0,x):x\in\mathbb F_{p^m}\}.
\]
Then $\mathcal S$ is called the Desarguesian $m$-spread of $V$. Each member of $\mathcal S$ has order $p^m$. Moreover, $U\cap W=\{0\}$ for all distinct $U,W\in\mathcal S$ and $V=\{0\}\cup_{U\in \mathcal S} (U\setminus \{0\})$. In particular, $|\mathcal S|=p^m+1$.
\end{definition}

The following construction is a standard source of PCP-type regular
PDSs; see Ma~\cite{S.L.MA} for background on PDSs of this type. 

\begin{proposition}
\label{prop:spread-PCP-PDS}
Let $\mathcal S$ be the Desarguesian $m$-spread of $V=\mathbb F_{p^m}^2$, 
and let $\mathcal A\subseteq\mathcal S$ satisfy $|\mathcal A|=\ell$. Define $P=\bigcup_{U\in \mathcal A}\bigl(U\setminus\{0\}\bigr)$. Then $P$ is an $\ell$-PCP type regular PDS in the additive group of $V$, with parameters
\[
\left(
p^{2m},\,
\ell(p^m-1),\,
p^m-2+(\ell-1)(\ell-2),\,
\ell(\ell-1)
\right).
\]
\end{proposition}

We now use a particular member of this family of PCP-type PDSs to construct an SDS.

\begin{thm}
\label{thm:PCP-derived-SDS}
Let $V=\mathbb F_{2^m}^2$ and $G=(V,+)$ be an additive abelian group. Let $\mathcal S$ be the Desarguesian $m$-spread of $V$. Choose $\mathcal A\subseteq \mathcal S$ with $|\mathcal A|=\ell=2^{m-1}-1$, and define $P=\cup_{U\in \mathcal A}(U\setminus \{0\})$. Then $D=P-0_G$ is an SDS in $G$ with parameters
\[
\left(
2^{2m},\,
2^{2m-1}-3\cdot 2^{m-1}+2,\,
2^{2m-2}-3\cdot 2^{m-1}+2
\right).
\]
\end{thm}

\begin{proof}
By Proposition~\ref{prop:spread-PCP-PDS}, the set \(P\) is an
\(\ell\)-PCP type regular partial difference set with parameters $\left(
2^{2m},k_P,\lambda_P,\mu_P
\right)$, where $k_P=\ell(2^m-1)$, $\lambda_P=2^m-2+(\ell-1)(\ell-2)$ and $\mu_P=\ell(\ell-1)$. It follows that $\lambda_P-\mu_P=2^m-2\ell$. For $\ell=2^{m-1}-1$, we have $\lambda_P-\mu_P=2$. Since every $U\in\mathcal S$ is a subgroup of $G$, the set $U\setminus\{0\}$ is closed under taking inverses. Hence $P^{(-1)}=P$. Since $P$ is a regular PDS with parameters $\left(
2^{2m},k_P,\lambda_P,\mu_P
\right)$, we have $P^2=
(k_P-\mu_P)0_G+(\lambda_P-\mu_P)P+\mu_PG$. For $D=P-0_G$, we therefore obtain 
\[
\begin{aligned}
DD^{(-1)}
&=(P-0_G)(P^{(-1)}-0_G)\\
&=PP^{(-1)}-P-P^{(-1)}+0_G\\
&=P^2-2P+0_G\\
&=(k_P-\mu_P+1)0_G+\mu_PG.
\end{aligned}
\]
Thus $D$ satisfies the defining group-ring identity of a $(v,k,\lambda)$-SDS with $k=k_P+1$ and $\lambda=\mu_P$. 
We calculate 
\[
k=k_P+1=(2^{m-1}-1)(2^m-1)+1=2^{2m-1}-3\cdot 2^{m-1}+2
\]
and
\[
\lambda=\mu_P=(2^{m-1}-1)(2^{m-1}-2)=2^{2m-2}-3\cdot 2^{m-1}+2.
\]
\end{proof}

\begin{example}
\label{ex:PCP-derived-database}
Among the open group-specific cases in the SDS database with $9\leq v\leq499$, Theorem~\ref{thm:PCP-derived-SDS}
turns two previously open cases into existence cases. Taking $m=3$, we obtain a $(64,22,6)$-SDS in $(\bF_8^2,+)\cong C_2^6$. Thus $\operatorname{SDS}(64,22,6,[2,2,2,2,2,2])$ exists. Taking $m=4$, the same construction gives $(256,106,42)$ in $(\mathbb F_{16}^2,+)\cong C_2^8$. Hence $\operatorname{SDS}(256,106,42,[2,2,2,2,2,2,2,2])$ exists.
\end{example}

\subsection{A \texorpdfstring{$(125,28,3)$}{(125,28,3)} SDS in  \texorpdfstring{$C_5^3$}{C5^3}}
\label{sec:125-construction}

In this section, we give an explicit construction of a $(125,28,3)$-SDS in the elementary abelian group $C_5^3$. The construction is based on quartic multiplicative
characters of $\mathbb F_{25}$ and $\mathbb F_5$. 

Throughout this section, multiplicative characters are extended to the whole finite field by setting their value at $0$ equal to $0$. The polynomial $f(x)=x^2+3$ is irreducible over $\bF_5$. Let $\alpha$ be a root of $f(x)$, and identify $\bF_{25}=\bF_5[\alpha]$. Thus $\alpha^2=-3=2$ in $\bF_5$. Let $\gamma=1+2\alpha$. It is easy to verify that $\gamma$ has multiplicative order $24$, and hence $\gamma$ is a primitive element of $\bF_{25}$. Let
\[
\rho:\mathbb F_{25}^*\longrightarrow
\{1,i,-1,-i\}
\]
be the quartic multiplicative character determined by $\rho(\gamma)=i$ over $\bF_{25}$. Since $2$ generates $\mathbb F_5^*$, let 
\[
\sigma:\mathbb F_5^*\longrightarrow
\{1,i,-1,-i\}
\]
be the quartic multiplicative character determined by $\sigma(2)=-i$ over $\bF_5$. For multiplicative characters $\chi,\psi$ of $\mathbb F_q$,
extended to $0$ by $\chi(0)=\psi(0)=0$, the Jacobi sum is
\[
J(\chi,\psi)
=
\sum_{x\in\mathbb F_q}\chi(x)\psi(1-x).
\]
With this notation, $J(\rho,\rho)$ is taken over $\mathbb F_{25}$, whereas $J(\sigma,\sigma)$ is taken over $\mathbb F_5$.

We first record the quartic Jacobi sums associated with $\rho$ and $\sigma$. 

\begin{lemma}
\label{lem:quartic-jacobi-25-5}
With the above choices of \(\rho\) and \(\sigma\), one has
\[
J(\rho,\rho)=3+4i
\]
and
\[
J(\sigma,\sigma)=-1+2i.
\]
Consequently,
\[
J(\rho,\rho)\,
\overline{J(\sigma,\sigma)}
=
5-10i,
\]
and in particular
\[
\operatorname{Re}
\left(
J(\rho,\rho)\,
\overline{J(\sigma,\sigma)}
\right)
=5.
\]
\end{lemma}

\begin{proof}
For the above choices of quartic characters, a direct computation gives
the following multiplicities:
\[
\begin{array}{c|ccccc}
z & 0 & 1 & -1 & i & -i\\
\hline
\#\{x\in\mathbb F_{25}:
\rho(x)\rho(1-x)=z\}
&2&7&4&8&4\\[1mm]
\#\{x\in\mathbb F_{5}:
\sigma(x)\sigma(1-x)=z\}
&2&0&1&2&0.
\end{array}
\]
Therefore,
\[
J(\rho,\rho)
=
7-4+(8-4)i
=
3+4i,
\]
whereas
\[
J(\sigma,\sigma)
=
-1+2i.
\]
It follows that
\[
\begin{aligned}
J(\rho,\rho)\,
\overline{J(\sigma,\sigma)}
&=
(3+4i)(-1-2i)\\
&=
5-10i.
\end{aligned}
\]
The desired real-part identity follows immediately.
\end{proof}

We now define the positive and negative supports of our signed
difference set. Let $G=(\mathbb F_{25},+)\times(\mathbb F_5,+)$.
 As an additive group, $(\mathbb F_{25},+)\cong C_5^2$ and $(\mathbb F_5,+)\cong C_5$. Hence $G\cong C_5^3$. Define
\[
\begin{aligned}
   P&=
\left\{
(x,y)\in
\mathbb F_{25}^*\times\mathbb F_5^*:
\rho(x)=\sigma(y)
\right\},\\
N&=\{0\}\times\mathbb F_5^*,
\end{aligned}
\]
and $D=P-N$. Clearly, $P\cap N=\varnothing$. 

\begin{thm}
\label{thm:125-28-3}
The signed set $D$ defined above is a $(125,28,3)$-SDS in $G\cong C_5^3$.
\end{thm}

\begin{proof}
We first determine the size of the support. For every fixed $y\in\mathbb F_5^*$, the value $\sigma(y)$ is a fourth root of unity. Since $\rho$ is a
multiplicative character of order $4$, each fiber of $\rho$ has cardinality $\frac{25-1}{4}=6$. Therefore $|P|=4\cdot6=24$. Moreover, $|N|=4$. Thus $k=|P|+|N|=28$. 

It remains to determine the nonprincipal character values of $D$. Let $\zeta=e^{2\pi i/5}$, 
and let $\Psi(x)=\zeta^{\operatorname{Tr}_{25/5}(x)}$, $x\in\mathbb F_{25}$, be the canonical additive character of $\mathbb F_{25}$. Let $\psi$ be the canonical additive character of $\bF_5$, defined by $\psi(y)=\zeta^y$ for $y\in\mathbb F_5$. Every additive character of $G$ has the form
\[
\Theta_{a,b}(x,y)
=
\Psi(ax)\psi(by),
\qquad
(a,b)\in\mathbb F_{25}\times\mathbb F_5.
\]
The principal character corresponds to $(a,b)=(0,0)$. We show that $\left|\Theta_{a,b}(D)\right|^2=25$ 
for every $(a,b)\ne(0,0)$.

\medskip
\noindent
\textbf{Case 1: $a=0$ and $b\ne0$.}

For each $y\in\mathbb F_5^*$, exactly six elements $x\in\mathbb F_{25}^*$ satisfy $\rho(x)=\sigma(y)$.
Hence
\[
\Theta_{0,b}(P)=6\sum_{y\in\mathbb F_5^*}\psi(by)=-6,
\]
because $b\ne0$. On the other hand,
\[
\Theta_{0,b}(N)
=
\sum_{y\in\mathbb F_5^*}\psi(by)
=-1.
\]
Consequently, $\Theta_{0,b}(D)=-6-(-1)=-5$, and therefore $|\Theta_{0,b}(D)|^2=25$. 

\medskip
\noindent
\textbf{Case 2: $a\ne0$ and $b=0$.}

For every $x\in\mathbb F_{25}^*$, there is a unique $y\in\mathbb F_5^*$ such that $\rho(x)=\sigma(y)$, since $\sigma$
 is bijective. Hence
\[
\Theta_{a,0}(P)=
\sum_{x\in\mathbb F_{25}^*}\Psi(ax)=-1.
\]
Also,
\[
\Theta_{a,0}(N)=4.
\]
Thus $\Theta_{a,0}(D)=-1-4=-5$, and again $|\Theta_{a,0}(D)|^2=25$.

\medskip
\noindent
\textbf{Case 3: $a\ne0$ and $b\ne0$.}

Indeed, since $P=\{(x,y)\in \mathbb F_{25}^*\times\mathbb F_5^*:
\rho(x)=\sigma(y)\}$, we have 
\[
1_P(x,y)=1
\quad\Longleftrightarrow\quad
\rho(x)\sigma(y)^{-1}=1.
\]
For a fourth root of unity $z$, the orthogonality relation 
\[
\frac14\sum_{j=0}^{3}z^j
=
\begin{cases}
1, & z=1,\\
0, & z\ne1
\end{cases}
\]
holds.  Applying this with $z=\rho(x)\sigma(y)^{-1}$ gives
\[
1_P(x,y)
=
\frac14
\sum_{j=0}^{3}
\left(\rho(x)\sigma(y)^{-1}\right)^j
=
\frac14
\sum_{j=0}^{3}
\rho^j(x)\sigma^{-j}(y).
\]
Therefore,
\[
\begin{aligned}
\Theta_{a,b}(P)
&=
\frac14
\sum_{j=0}^{3}
\left(
\sum_{x\in\mathbb F_{25}^*}
\rho^j(x)\Psi(ax)
\right)
\left(
\sum_{y\in\mathbb F_5^*}
\sigma^{-j}(y)\psi(by)
\right).
\end{aligned}
\]

For a multiplicative character $\chi$ of $\bF_q$, where $q\in\{5,25\}$, let $G_q(\chi)$ denote its Gauss sum with respect to $\psi$ when $q=5$ and $\Psi$ when $q=25$. Explicitly, 
\[
G_5(\chi)
=
\sum_{y\in\mathbb F_5^{*}}\chi(y)\psi(y),
\qquad
G_{25}(\chi)
=
\sum_{x\in\mathbb F_{25}^{*}}\chi(x)\Psi(x).
\]

For \(a\ne0\) and \(b\ne0\), we have
\[
\sum_{x\in\mathbb F_{25}^*}
\rho^j(x)\Psi(ax)
=
\rho^{-j}(a)G_{25}(\rho^j)
\]
and
\[
\sum_{y\in\mathbb F_5^*}
\sigma^{-j}(y)\psi(by)
=
\sigma^j(b)G_5(\sigma^{-j}).
\]

Set 
\[
\tau=\rho^{-1}(a)\sigma(b)
\in\{1,i,-1,-i\}.
\]
Since the Gauss sum of the trivial multiplicative character is $-1$, the $j=0$ term is equal to $1$. Hence
\[
4\Theta_{a,b}(P)
=
1+\tau A+\tau^2B+\tau^3C,
\]
where
\[
A=
G_{25}(\rho)G_5(\sigma^{-1}),
\]
\[
B=
G_{25}(\rho^2)G_5(\sigma^2),
\]
and
\[
C=
G_{25}(\rho^{-1})G_5(\sigma).
\]

We next record the relations among $A,B,C$. Since $-1=\gamma^{12}$, we have $\rho(-1)=1$. On the other hand, $-1=2^2$ in $\mathbb F_5^*$,  and hence $\sigma(-1)=(-i)^2=-1$. Using the standard identity $G_q(\chi^{-1})=\chi(-1)\overline{G_q(\chi)}$, we obtain $C=-\overline A$.

For a nontrivial quadratic character $\eta$ of $\bF_q$, the standard Gauss sum identities give
\[
|G_q(\eta)|^2=q,
\qquad
\overline{G_q(\eta)}=\eta(-1)G_q(\eta).
\]
Consequently, 
\[
G_q(\eta)^2=\eta(-1)q.
\]
Since $\rho^2$ and $\sigma^2$ are nontrivial quadratic characters with $\rho^2(-1)=\sigma^2(-1)=1$, both $G_{25}(\rho^2)$ and $G_5(\sigma^2)$ are real, and $G_{25}(\rho^2)^2=25$ and $G_5(\sigma^2)^2=5$. Therefore $B$ is real and $B^2=25\cdot5=125$. Moreover, since $\rho$ and $\sigma^{-1}$ are nontrivial,
\[
|A|^2
=
|G_{25}(\rho)|^2|G_5(\sigma^{-1})|^2
=
25\cdot5
=
125.
\] 

For a multiplicative character $\chi$ such that both $\chi$ and $\chi^2$ are nontrivial, the Gauss--Jacobi identity $G_q(\chi)^2=G_q(\chi^2)J(\chi,\chi)$
 gives
\[
\begin{aligned}
A^2
&=
G_{25}(\rho)^2
G_5(\sigma^{-1})^2\\
&=
B\,
J(\rho,\rho)
J(\sigma^{-1},\sigma^{-1}).
\end{aligned}
\]
Since $J(\sigma^{-1},\sigma^{-1})=\overline{J(\sigma,\sigma)}$, Lemma~\ref{lem:quartic-jacobi-25-5} yields $A^2=B(5-10i)$. Consequently, $\operatorname{Re}(A^2)=5B$. Write $A=X+iY$ for some $X,Y\in\mathbb R$. Then $X^2+Y^2=125$ and $X^2-Y^2=5B$. 
Hence $2X^2=125+5B$ and $2Y^2=125-5B$.

Since \(b\ne0\),
\[
\Theta_{a,b}(N)
=
\sum_{y\in\mathbb F_5^*}\psi(by)
=-1.
\]
Thus
\[
\Theta_{a,b}(D)
=
\Theta_{a,b}(P)+1,
\]
and therefore $4\Theta_{a,b}(D)=5+\tau A+\tau^2B-\tau^3\overline A$.

There are four possibilities for $\tau$.

If $\tau=1$, then $4\Theta_{a,b}(D)=5+B+2iY$. Hence
\[
\begin{aligned}
16|\Theta_{a,b}(D)|^2
&=
(5+B)^2+4Y^2\\
&=
25+10B+B^2
   +2(125-5B)\\
&=
25+125+250\\
&=400,
\end{aligned}
\]
where we have used $B^2=125$. Thus $|\Theta_{a,b}(D)|^2=25$.

If $\tau=-1$, then $4\Theta_{a,b}(D)=5+B-2iY$, and the same computation gives $|\Theta_{a,b}(D)|^2=25$. 

If $\tau=i$, then $4\Theta_{a,b}(D)=5-B+2iX$. Therefore, 
\[
\begin{aligned}
16|\Theta_{a,b}(D)|^2
&=
(5-B)^2+4X^2\\
&=
25-10B+B^2
   +2(125+5B)\\
&=400,
\end{aligned}
\]
and hence $|\Theta_{a,b}(D)|^2=25$. 

Finally, if $\tau=-i$, then $4\Theta_{a,b}(D)=5-B-2iX$, and again $|\Theta_{a,b}(D)|^2=25$.

We have therefore proved that $|\Theta(D)|^2=25$ for every nonprincipal character $\Theta$ of $G$.

Since $k=28$, the character criterion for signed difference sets gives $k-\lambda=25$. Thus $\lambda=28-25=3$. Consequently, $DD^{(-1)}=25\cdot 0_G+3G$, and hence $D$ is a $(125,28,3)$-SDS in $G\cong C_5^3$.
\end{proof}

\begin{remark}
\label{rem:125-open}
The ambient group in Theorem~\ref{thm:125-28-3} is the elementary abelian group $C_5^3$.  According to the current SDS database, the existence of a $(125,28,3)$-SDS in $C_5^3$ is listed as open.  Thus
Theorem~\ref{thm:125-28-3} settles this existence case.
\end{remark}

\begin{remark}
The construction is genuinely signed: both the positive and negative supports are nonempty, with $|P|=24$ and $|N|=4$. In particular, the construction is not an ordinary DS. Moreover, the defining sets are obtained directly from quartic multiplicative characters rather than by converting a PDS into an SDS.
\end{remark}

\section{Conclusion}\label{sec_con}
In this paper, we established several new nonexistence criteria for
SDSs in finite abelian groups. The self-conjugate prime obstruction adapts a classical character-sum method, while the support-refined quotient method makes use of the relation between signed intersection numbers and the numbers of nonzero coefficients on cosets. This leads to the $C_2$-quotient, near-full-support, and small-defect obstructions. Applied cumulatively to the open group-specific cases in the signed difference set database with $9\leq v\leq499$, these criteria rule out $45{,}361$ cases.

We also obtained two existence results. The first uses PCP-type
regular partial difference sets arising from Desarguesian spreads to
construct an infinite family of SDSs in elementary abelian $2$-groups.  The second uses quartic multiplicative characters to
construct an explicit SDS in $C_5^3$. Together, these
constructions settle three previously open cases. Possible directions for further research include extending the quotient methods to more general groups and finding additional combinatorial or character-based constructions for the remaining open cases.

\section{Acknowledgements}
This project was supported by the National Natural Science Foundation of China under Grant 12301429.

\bibliographystyle{IEEEtran}
\bibliography{sds}
\end{document}